\documentclass[a4paper,11pt]{article}
\usepackage{amsmath}
\usepackage{dsfont}
\usepackage{empheq}
\usepackage{graphicx}
\usepackage{amsfonts}
\usepackage{amssymb}
\usepackage[english]{babel}
\usepackage{hyperref}
\usepackage[T1]{fontenc}
\usepackage{color}
\usepackage{fancyhdr}
\usepackage[dvips,lmargin=2.5cm,rmargin=2.5cm,tmargin=2.5cm,bmargin=2.5cm]{geometry}
\usepackage{multirow}
\usepackage{arydshln}
\usepackage{ntheorem}
\usepackage{xspace}
\usepackage[round]{natbib}
\usepackage[normalem]{ulem}
\allowdisplaybreaks

\title{\large\textbf{An exact copositive programming formulation for the Discrete Ordered Median Problem: Extended version}}

\newtheorem{thm}{Theorem}[section]
\newtheorem{propo}{Proposition}[section]
\newtheorem{lemma}{Lemma}[section]

\newtheorem*{proof}{Proof}

\newcommand{\DOMP}{DOMP\xspace}

\newcommand{\y}{\mathcal{O}\xspace}
\newcommand{\X}{\mathcal{X}\xspace}
\newcommand{\red}{\color{red}}
\newcommand{\rev}[1]{#1}

\author{J. Puerto\thanks{IMUS. Universidad de
Sevilla, Spain. e-mail:puerto@us.es}}

\date{\today}

\begin{document}
\maketitle
\begin{abstract}
This paper presents a first continuous, linear, conic formulation for the Discrete Ordered Median Problem (\DOMP). Starting from a binary, quadratic formulation in the original space of location and allocation variables that are common in Location Analysis (L.A.), we prove that there exists a transformation of that formulation, using the same space of variables, that allows us to cast \DOMP as a continuous linear problem in the space of completely positive matrices. This is the first positive result that states equivalence between  the family of continuous convex problems and some hard problems in L.A. The result is of theoretical interest because it allows us to share the tools from continuous optimization to shed new light into the difficult combinatorial structure of the class of ordered median problems.
\end{abstract}

\medskip

\section{Introduction}
The discrete ordered median problem (\DOMP) represents a generalization of several well-known discrete location problems, such as p-median, $p$-center or $(k_{1}+k_{2})-$trimmed mean, among many others. \DOMP provides a very flexible tool in accommodating actual aspects, as subsidized or compensation costs,  in models of Logistics and Location, see \citet{Puerto1999} and \citet{Nickel2005}.  \rev{The} ordered median objective function computes ordered weighted averages of vectors. \rev{In the case when that objective function is applied to location problems, those vectors correspond to distances or allocation costs from clients to service facilities.} This way many different objective functions are cast within the same framework and the same tools can be used to solve problems with apparently different structure.  The problem was introduced in \citet{Nickel2001} and later studied by \cite{Kalcsics2003}, \cite{Nickel2005}, \citet{Boland2006}, \cite{Stanimirovic2005}, \cite{Marin2009} and \citet{BHP2014}  among many other papers. \DOMP is an NP-hard problem as an extension of the p-median problem.

\cite{Nickel2001} first presented a quadratic integer programming formulation for the \DOMP. However, no further attempt to deal directly with this formulation was ever considered. Furthermore, that approach was never exploited in trying to find alternative reformulations or bounds; instead several linearizations in different spaces of variables have been proposed to solve \DOMP some of them  being rather promising, \citet{Marin2009} and \citet{Labbe2016}.

Motivated by the recent advances in conic optimization and the new tools that this branch of mathematical programming has provided for developing bounds and approximation algorithms for NP-hard problems, as for instance \rev{the max-cut problem and the quadratic assignment problem (QAP)} among others \citep{Burer2015}, we would like to revisit that earlier approach to \DOMP with the aim of proposing an alternative reformulation \rev{of the problem} as a continuous, linear conic optimization problem. Our interest is mainly theoretical and  tries to borrow tools from continuous optimization to be applied in some discrete problems in the field of \rev{Location Analysis (L.A.).}  To the best of our knowledge reformulations of that kind  have never been studied before for \DOMP nor even in the wider field of L.A.

The goal of this paper is to prove that a natural binary, quadratically constrained, quadratic formulation for \DOMP \,  \rev{ admits a reformulation as a continuous, linear optimization problem} over the cone of completely positive matrices $\mathcal{C}^*$.  Recall that a symmetric matrix $M\in \mathbb{R}^{n\times n}$ is called completely positive if it can be written as $M=\sum_{i=1}^k x_i x_i^t$ for some finite $k\in \mathbb{N}$ and \rev{$x_i\in \mathbb{R}^n_+$} for all $i=1,\ldots,k$.

The paper is organized as follows. Section \ref{ss:def} formally defines the ordered median problem and sets its elements. Section \ref{ss:sorting} is devoted to describe a folk result that formulates the problem of sorting numbers as a feasibility binary linear program. This is used later as a building block to present the binary quadratic, quadratically constrained formulation of \DOMP. Section \ref{s:codomp} contains the main result in this paper: \DOMP is equivalent to a continuous, linear conic \rev{optimization} problem. Obviously, there are no shortcuts and the problem remains $NP$-hard but it allows us to shed some light onto the combinatorics of this difficult discrete location problem. Moreover, it permits \rev{to apply also some tools from continuous optimization on} the area of discrete L.A. The last section, namely Section \ref{s:remarks} contains some conclusions and pointers to future research. 


\section{Definition and formulation of the problem}
\subsection{Problem definition \label{ss:def}}
Let $S=\{1,...,n\}$ denote the set of $n$ sites. Let $C=(c_{j\ell})_{j,\ell=1,...,n}$ be \rev{a given nonnegative $n\times n$ cost matrix,} where
$c_{j\ell}$ denotes the cost of satisfying demand point (client) $j$ from a facility
located at site $\ell$. We also assume the so called, \textit{free self-service} situation, namely $c_{jj}=0$ for all $j=1,\ldots,n$. Let $p<n$ be the desired number of facilities to be located at the candidate sites. A solution to the facility location problem is given by a set $\X\subseteq S$ of $p$ sites. The set of solutions to the location problem is therefore finite, although exponential in size, and it coincides with the $n \choose p$ subsets of size $p$ of $S$.

We assume, that each new facility has unlimited capacity. Therefore, each client $j$ will be allocated to a facility located at site $\ell$ of $\X$ with lowest cost, i.e.
$$
c_{j}=c_{j}(\X):=\min_{\ell\in \X} c_{j\ell}.
$$

The costs for supplying clients, $c_{1}(\X), . . . , c_{n}(\X)$, are sorted in nondecreasing order. We define $\sigma_{\X}$ to be a permutation on $\{1, . . . , n\}$ for which the inequalities
$$
c_{\sigma_{\X}(1)}(\X)\leq\cdots\leq c_{\sigma_{\X}(n)}(\X)
$$
hold.

Now, for any nonnegative vector  $\lambda\in \mathbb{R}^n_+$, \DOMP consists of finding $\X^*\subset S$ with $|\X^*|=p$ such that:
$$\sum_{k=1}^n \lambda_k c_{\sigma_{\X^*}(k)}(\X^*)=\min _{\X\subset S, |\X|=p} \sum_{k=1}^n \lambda_k c_{\sigma_{\X}(k)}(\X).$$

\subsection{Sorting as an integer program and a binary quadratic programming formulation of \DOMP\label{ss:sorting}}
For the sake of readability, we describe in the following a folk result that allows us to understand better the considered binary quadratic programming formulation for \DOMP.

Let us assume that we are given $n$ real numbers $r_1, \ldots, r_n$ and that we are interested in finding its sorting in nondecreasing sequence as a solution of a mathematical program. One natural way to do it is to identify the permutation that ensures such a sorting.

We can introduce the following \textit{ordering} variables
$$
P_{jk} =\left\{
        \begin{array}{ll}
          1\quad \text{ if the number $r_j$ is sorted in the $k$-th position}, \; j,k =1,\ldots,n, \\
          0\quad \text{ otherwise}.
        \end{array}
      \right.
$$
Next, we can consider the following feasibility problem:
\begin{align}
  \min \quad & 1 \nonumber \tag{\textbf{SORT}}\\
  \mbox{s.t. } & \sum_{j=1}^{n} P_{jk} = 1&   \forall k&=1,\dots,n \label{sort:const1}\\
 & \sum_{k=1}^{n} P_{jk} = 1&   \forall j&=1,\dots,n \label{sort:const2}\\
 &  \sum_{j=1}^n P_{jk} r_j   \leq \sum_{j=1}^n P_{j,k+1}
 r_j& \forall k & =1,\ldots, n-1  \label{sort:const3} \\
 & P_{jk} \in \{0,1\}&    \forall j,k&=1,\dots,n. \nonumber
\end{align}
Clearly, constraints (\ref{sort:const1}) and (\ref{sort:const2}) model permutations and constraint (\ref{sort:const3}) ensures the appropriate order in any feasible solution since the number in sorted position $k$ must be smaller than or equal to the one in position $k+1$.

Now, we can insert in this formulation the elements of \DOMP. Indeed, we have to replace the given real numbers $r_1,\dots,r_n$ by the allocation costs induced by the location problem. Let $X$ and $\y$ be the natural allocation and location variables in the location problem, namely
$$
X_{j\ell} =\left\{
        \begin{array}{ll}
          1\quad \text{  if the demand point $j$ goes to the facility $\ell$}, \; j,\ell =1,\ldots,n,\\
          0\quad \text{ otherwise},
        \end{array}
      \right.
$$
\noindent and
$$
\y_{\ell} =\left\{
        \begin{array}{ll}
          1\quad \text{ if a new facility is opened at site $\ell$,}\;  \ell=1,...,n, \\
          0\quad \text{ otherwise}.
        \end{array}
      \right.
$$

Based on the above variables, it follows that the cost induced by the allocation of the demand point $j$ will be given by $\sum_{\ell =1}^n c_{j\ell}X_{j\ell}$. Thus, if we want a valid formulation of \DOMP we can embed the location elements into the SORT formulation. This results in the following formulation (see \citep{Nickel2001}):

\begin{align}
\min & \displaystyle \sum_{k=1}^n \lambda_k \sum_{j=1}^n P_{jk}\sum_{\ell =1} ^n c_{j \ell} X_{j\ell}  \label{pro:comb-ofp}\tag{MIQP-DOMP}\\
s.t. & \displaystyle\sum_{j=1}^n P_{jk}=1 & \forall k=1,\dots,n \label{pro:comb1p}\\
&\displaystyle\sum_{k=1}^n P_{jk}=1 & \forall j=1,\dots,n \label{pro:comb2p}\\
&\displaystyle\sum_{j=1}^{n}\left(\sum_{\ell=1}^{n}c_{j\ell} X_{j\ell}\right)P_{jk} \le \sum_{j=1}^{n}\left(\sum_{\ell=1}^{n}c_{j\ell} X_{j\ell}\right)P_{jk+1},& \forall k=1,\ldots,n-1, \label{pro:comb3p}\\
&\displaystyle\sum_{\ell=1}^n \y_{\ell}=p \label{pro:comb4p}\\
&\displaystyle\sum_{\ell=1}^n X_{j\ell}=1 & \forall j=1,\dots,n \label{pro:comb5p}\\
&\y_{\ell}\geq X_{j\ell} & \forall j,\ell=1,\dots,n \label{pro:comb6p}\\
&\rev{P_{jk},X_{j\ell},\y_j \in\{0,1\},} & \forall
j,k,\ell=1,\dots,n \label{pro:combdp}
\end{align}

Clearly, the objective function accounts for the ordered weighted sum of the allocation costs so that the cost sorted in position $k$, in a given feasible solution, is multiplied by $\lambda_k$ for all $k=1,\ldots,n$. Constraint (\ref{pro:comb1p}) states that each position in the sequence from $1$ to $n$ will be occupied by one of the realized allocation costs; whereas constraint (\ref{pro:comb2p}) ensures that each allocation cost will be in one position between $1,\ldots,n$. Constraint (\ref{pro:comb3p}) states that the cost in position $k$ must be less than or equal to the one in position $k+1$. With constraint (\ref{pro:comb4p}) it is fixed that $p$ facilities are open. Constraint (\ref{pro:comb5p}) forces that each demand point is allocated to one facility and constraint (\ref{pro:comb6p}) ensures that demand points can be assigned only to open facilities. Finally, (\ref{pro:combdp}) sets the range of the variables.

Based on our discussion, the above quadratically constrained, quadratic objective function is a valid formulation for \DOMP.
We note in passing that one could relax the binary definition of the variables $\y_j\in \{0,1\}$ to simply nonnegativity, namely $\y_j\ge 0$, \rev{since boundedness of $\y$ is ensured by (\ref{pro:comb4p}) and its binary character is forced by (\ref{pro:comb6p}). Therefore, in what follows $\y$ variables will be considered just nonnegative.}

\section{A completely positive reformulation of \DOMP\label{s:codomp}}
The goal of this section is to developed a new  reformulation for \DOMP as a continuous linear \rev{optimization} problem in the cone of completely positive matrices $\mathcal{C}^*$. For this reason, we consider $X$, $P$ and $\y$ as the original variables of the problem. These variables are standard in the already known formulations for \DOMP.

For ease of presentation we use the following notation. By $P_{i\cdot}$ and $X_{i\cdot}$ for all $i=1,\ldots,n$, we refer, respectively, to the $i-th$ row of $P$ and $X$. That is, \rev{the row vector} $P_{i\cdot}=(P_{i1},\ldots,P_{in})$ and \rev{the row vector} $X_{i\cdot}=(X_{i1},\ldots,X_{in})$. We will also follow the following convention. We refer to original variables in capital letters. Slack variables used to transform inequalities to equations, in the formulations, will be denoted by the Greek letters $\zeta$, $\xi$, $\eta$. Finally, \rev{the elements of} matrix variables will be denoted by small letters. For instance, if $U=(u_{ij})$ is a matrix variable we refer to its elements as $u_{ij}$. In the following the \textit{vec} operator stacks in a column vector the columns of a matrix and the \textit{rvec} operator stacks in a column vector the rows of a matrix. Clearly,  $rvec(A)=vec(A^T)$, i.e. \textit{rvec} acting on a matrix $A$ is equivalent to $vec$ acting on the transpose of $A$. \rev{The $Diag$ operator maps a $n-$vector to a $n\times n$ diagonal matrix and the $diag$ operator extracts in a $n$-vector the diagonal of a $n\times n$-square matrix.} We follow the convention that all \rev{single-index} variables in our formulations are column vectors.

In the sequel rather than considering the standard version of \DOMP in the literature of L.A., described above, we will consider an extended version of that problem because it permits to include more realistic aspects in the model. These modeling aspects come from the interaction between the different elements in the problems. \rev{That is, this extended model} includes some extra interaction terms among  the ordering variables ($P_{jk}P_{j'k'}$) and the allocation variables ($X_{j\ell}X_{pq}$). Let $D=(d_{jkj'k'})\in \mathbb{R}^{n^2\times n^2}$, where $d_{jkj'k'}$ is the interaction cost \rev{of} sorting demand point $j$ in position $k$ and demand point $j'$ in position $k'$. Moreover, let $H=(h_{j\ell pq})\in \mathbb{R}^{n^2\times n^2}$, where $h_{j\ell pq}$ is the cost incurred due to the allocation of demand points $j$ and $p$ to facilities $\ell$ and $q$, respectively, see e.g. \citet{BURKARD1984283}. The reader should observe that in the standard version of \DOMP in the literature the cost matrices $D$ and $H$ satisfy \rev{$D=H=\Theta$, where $\Theta$ is the zero matrix.}

In spite of the greater complexity coming from including new quadratic terms in the objective function, the tools that we will use in the following  can be accommodated to the extended version of the problem with no extra effort.

The above discussion means that in what follows we will consider the following objective function for \ref{pro:comb-ofp}.

\begin{align}
\min & \displaystyle \sum_{k=1}^n \lambda_k \sum_{j=1}^n P_{jk}\sum_{\ell =1} ^n c_{j \ell} X_{j\ell} +1/2 \sum_{j=1}^n\sum_{k=1}^n\sum_{j'=1}^n\sum_{k'=1}^n  d_{jkj'k'}P_{jk}P_{j'k'} \nonumber \\ & +1/2\sum_{j=1}^n\sum_{\ell=1}^n\sum_{p=1}^n\sum_{q=1}^n  h_{j\ell pq}X_{j\ell}X_{pq} \label{pro:comb-ofp-ext}
\end{align}

Clearly, the objective function (\ref{pro:comb-ofp-ext}) accounts for the ordered weighted sum of the allocation costs so that the cost sorted in position $k$, in a given feasible solution, is multiplied by $\lambda_k$ for all $k=1,\ldots,n$, plus the interaction costs induced by the ordering and the allocations. This objective function replaces the usual one in \ref{pro:comb-ofp} and allows us a more general analysis of the problem.

Our goal is to find new ways to represent the problem that \rev{may shed some light into the structure of the difficult combinatorics of \DOMP that combine elements from the p-median, the quadratic assignment (QAP) and the permutation polytopes}.  We  show in the following that it is possible to cast \DOMP within the field of linear continuous conic  optimization and we derive an explicit reformulation.

Furthermore, we emphasize that this  result is not straightforward. \rev{Contrary to the case of  quadratic objective and linear constrained problems with some binary variables, where one can apply  Burer's results \citep[Theorem 2.6]{Burer09}; the above formulation is quadratic in the objective, in the constraints and also includes binary variables; and in our case the most recent, known sufficient conditions for obtaining conic reformulations do not directly apply \citep{Burer2012,BURER2012203,Pena2015,Bai2016}.} In all cases, those results require some nonnegativity condition of the considered quadratic constraints on the feasible region of the problem. \rev{This condition cannot be ensured in our problem since the constraints (\ref{pro:comb3p}), namely $\displaystyle\sum_{j=1}^{n}\left(\sum_{\ell=1}^{n}c_{j\ell} X_{j\ell}\right)P_{jk} \le \sum_{j=1}^{n}\left(\sum_{\ell=1}^{n}c_{j\ell} X_{j\ell}\right)P_{jk+1},\; \forall k=1,\ldots,n-1$, can assume any sign on the feasible solutions of the problem.}  In spite of that, we prove in our main result that such a \rev{reformulation} does exist for \DOMP. Thus, it links this problem to the general theory of convex conic programming as a new instance of this broad class of problems. \rev{As a result, any new development arising in that optimization field (convex conic programming), could be readily  transferred  to the combinatorial problem advancing in its better understanding.} The interested reader is referred to \cite{Dur10}, \cite{Bomze12}, \cite{Burer2012} and \cite{XuBurer2016} for a list of \rev{combinatorial} problems that also fits to this framework.

In our approach we first reformulate \ref{pro:comb-ofp} as another quadratically constrained, quadratic problem that is instrumental in the construction of \rev{the final copositive reformulation for \DOMP, that is the main goal of this paper.}

\rev{We start by slightly modifying the formulation above to replace the set of inequalities (\ref{pro:comb6p}) by an equivalent set with smaller cardinality.}

\begin{propo} \label{prop:surrogate}
Let us assume that the set of inequalities (\ref{pro:comb1p})-(\ref{pro:combdp}) hold. Then, the set of inequalities (\ref{pro:comb6p}) can be replaced by
\begin{equation}
\rev{\sum_{j=1}^n X_{j\ell}-(n-p +1)\y_{\ell}  \le 0,\;   \quad \forall \ell=1,\ldots,n  } \label{pro:comb6pmod}
\end{equation}
resulting in the same feasible set.
\end{propo}
\begin{proof}
{\rm Indeed, let us assume that (\ref{pro:comb6p}) holds together with the remaining inequalities in the formulation \ref{pro:comb-ofp}. We shall prove that (\ref{pro:comb6pmod}) also holds.
We sum $X_{j\ell}\le \y_{\ell}$ for all $j=1...,n$ which results in
\begin{equation}
\sum_{j=1}^n X_{j\ell}\le n\y_{\ell}. \label{expl:1}
\end{equation}

Next, we observe that since we assume all costs to be nonnegative and there is no capacity constraints for the allocation of demand points to facilities then for each $\y_{\hat \ell}=1$ the corresponding $X_{\hat \ell \hat \ell}=1$ as well; that is, the demand point at $\hat \ell$ will be served by the facility at the same location. Since there are $p$ open facilities, there exist $\y_{\ell_1}=1,\ldots,\y_{\ell_p}=1$. This implies that in any sum of the form $\sum_{j=1}^n X_{j\ell}\le n\y_{\ell}$ there must be at least $p-1$ variables $X_{\ell_j\ell}=0$, {\red if} $\ell_j\neq \ell$, since by (\ref{pro:comb4p}) there are, at least, $p-1$ open facilities different from $\ell$ (even if facility at $\ell$ happens to be open).

The above argument implies that, whenever $\y_{\ell}=1$ in (\ref{expl:1}), we can subtract $p-1$ units from the right-hand-side of the inequality,
(obviously, this also holds if $\y_\ell=0$) resulting in (\ref{pro:comb6pmod}).

Conversely, let us assume that (\ref{pro:comb6pmod}) replaces (\ref{pro:comb6p}) in the above formulation. We have to prove that (\ref{pro:comb6p}) still holds. Clearly if $\y_{\ell}=0$ then it follows that $X_{j\ell}=0$ for all $j=1,\ldots,n$. Besides, if $X_{j\ell}$=1 for some $j$ then $\y_{\ell}=1$ because otherwise (\ref{pro:comb6pmod}) would not hold. Finally, if $\y_{\ell}=1$ or $X_{j\ell}=0$ for all $j=1,\ldots,n$, (\ref{pro:comb6p}) holds trivially. This proves that (\ref{pro:comb6pmod}) implies (\ref{pro:comb6p}) in the formulation above.}

 \hfill $\Box$
\end{proof}

\rev{In our development we follow the approach by \cite{Burer09} but we need to introduce modifications to go around the problematic quadratic constraint (\ref{pro:comb3p}) that does not fit within that framework. In order to do that, } we introduce an additional set of variables $W_k$, $k=1,\dots,n$ which will represent the value of the cost that is sorted in position $k$. This means that $W_k=\sum_{j=1}^{n}\left(\sum_{\ell=1}^{n}c_{j\ell} X_{j\ell}\right)P_{jk}$ and although these variables are redundant they will simplify the presentation in our approach. \rev{Adding these equations for all $k=1,\ldots,n$, it results in another valid constraint, namely $ \sum_{k=1}^n W_k= \sum_{j=1}^n\sum_{\ell=1}^n c_{j\ell} X_{j\ell}$, that we will also include in our formulation.}

The following result provides a reformulation of \ref{pro:comb-ofp}  that is an important building block in the proof of the main result.

\begin{lemma}
\rev{Problem  \ref{pro:comb-ofp2} is a reformulation of Problem \ref{pro:comb-ofp}.}
\begin{align}
\min & \displaystyle \sum_{k=1}^n \lambda_k W_k +1/2\sum_{j=1}^n\sum_{k=1}^n\sum_{j'=1}^n\sum_{k'=1}^n  d_{jkj'k'}P_{jk}P_{j'k'} \nonumber \\ & +1/2\sum_{j=1}^n\sum_{\ell=1}^n\sum_{p=1}^n\sum_{q=1}^n  h_{j\ell pq}X_{j\ell}X_{pq} \label{pro:comb-ofp2}\tag{MIQP1-DOMP}\\
s.t. &  \eqref{pro:comb1p},\eqref{pro:comb2p}, \eqref{pro:comb4p}, \eqref{pro:comb5p},  \nonumber \\
& \rev{\sum_{j=1}^n X_{j\ell}-(n-p +1)\y_{\ell} +\zeta_\ell = 0,\;   \quad \forall \ell=1,\ldots,n   \label{pro:comb6pmodeq}}\\
& W_k- \displaystyle\sum_{\ell=1}^n c_{j\ell}\, X_{j\ell}+\displaystyle\sum_{\ell =1}^n
c_{j\ell}(1-P_{jk}) - \eta_{jk}=0,\quad \forall
j,k=1,\dots,n \label{pro:comb1p8}\\
& \displaystyle W_k=\sum_{j=1}^n P_{jk}\sum_{\ell =1} ^n c_{j \ell} X_{j\ell},\quad \forall k=1,\ldots,n \label{pro:comb1p2}\\
& \sum_{j=1}^n\sum_{k=1}^n (P_{jk}-P_{jk}^2)=0 \label{pro:comb2p2}\\
& \rev{\sum_{j=1}^n\sum_{\ell=1}^n (X_{j\ell}-X_{j\ell}^2)=0 \label{pro:combq2p2}}\\
& W_k - W_{k+1} +\xi_k=0,\quad  \forall
k=1,\dots,n-1 \label{pro:comb1p3}\\
& \rev{\displaystyle\sum_{k=1}^n W_k-\displaystyle\sum_{j=1}^n \displaystyle\sum_{\ell =1}^n
c_{j\ell}\, X_{j\ell}=0, \label{pro:comb1p7}}\\
&P_{jk},X_{j\ell}\in\{0,1\}, \; \y_j, W_k, \xi_j,\eta_{jk}, \zeta_\ell \ge 0, \quad \forall
j,k,\ell=1,\dots,n \label{pro:combdp1}
\end{align}

\end{lemma}
\begin{proof}
{\rm
In order to get the formulation in the statement of the lemma, we begin by considering an intermediate reformulation of \ref{pro:comb-ofp} which results by substituting constraints \eqref{pro:comb6p} by the equivalent family  $\sum_{j=1}^n X_{j\ell}- (n-p+1)\y_{\ell}+\zeta_l=0$, $\ell= 1,\ldots,n$, namely \eqref{pro:comb6pmodeq}, that was already proved in Proposition \ref{prop:surrogate} and augmenting the variables $W_k=\sum_{j=1}^n P_{jk}\sum_{\ell =1} ^n c_{j \ell} X_{j\ell}$, for all $k=1,\ldots,n$ \eqref{pro:comb1p2}. \rev{In addition, we add three redundant constraints. The first two are quadratic, namely  $\sum_{j=1}^n\sum_{k=1}^n (P_{jk}-P_{jk}^2)=0$ \eqref{pro:comb2p2} and $\sum_{j=1}^n\sum_{\ell=1}^n (X_{j\ell}-X_{j\ell}^2)=0$ \eqref{pro:combq2p2};  and the third is linear $\sum_{k=1}^n W_k-\sum_{j=1}^n \sum_{\ell =1}^n c_{j\ell}\, X_{j\ell}=0$ \eqref{pro:comb1p7}. These three constraints will be instrumental in our developments.} All the above changes do not alter the optimal solutions of \ref{pro:comb-ofp}.

We write the new problem explicitly for the sake of readability. Starting from \ref{pro:comb-ofp}, with the new objective function (\ref{pro:comb-ofp-ext}), we perform the transformation described above and the new formulation is:
\begin{align}
\min & \displaystyle \sum_{k=1}^n \lambda_k W_k +1/2\sum_{j=1}^n\sum_{k=1}^n\sum_{j'=1}^n\sum_{k'=1}^n  d_{jkj'k'}P_{jk}P_{j'k'} \nonumber \\ & +1/2\sum_{j=1}^n\sum_{\ell=1}^n\sum_{p=1}^n\sum_{q=1}^n  h_{j\ell pq}X_{j\ell}X_{pq} \label{pro:comb-ofp1}\tag{MIQP2-DOMP}\\
s.t. &  \eqref{pro:comb1p},\eqref{pro:comb2p}, \eqref{pro:comb3p}, \eqref{pro:comb4p}, \eqref{pro:comb5p}, \eqref{pro:comb6pmodeq},\eqref{pro:comb1p2}, \eqref{pro:comb2p2}, \eqref{pro:combq2p2}, \eqref{pro:comb1p7}, \nonumber \\
&P_{jk},X_{j\ell},\y_j\in\{0,1\}, \; W_k,\zeta_{\ell} \ge 0 & \forall
j,k,\ell=1,\dots,n \label{pro:combdp1}
\end{align}
\ref{pro:comb-ofp1} is a mixed-$\{0,1\}$ quadratically constrained, quadratic objective \rev{problem}. From this formulation, we can obtain the one in the statement of the lemma by:  1) reformulating constraints (\ref{pro:comb3p}) as $W_k- W_{k+1}-\xi_k=0$ for all $k=1,\ldots,n-1$, \eqref{pro:comb1p3},
and; 2) \rev{adding to the formulation  the valid inequalities}  (\ref{pro:comb1p8}), i.e
$
W_k- \displaystyle\sum_{\ell=1}^n c_{j\ell}\, X_{j\ell}+\displaystyle\sum_{\ell =1}^n
c_{j\ell}(1-P_{jk}) - \eta_{jk}=0,\quad \forall
j,k=1,\dots,n$.

We observe that the constraint (\ref{pro:comb1p8}) is valid by the following argument. First of all, if $P_{jk}=0$ then (\ref{pro:comb1p8}) reduces to $W_k\ge \sum_{\ell=1}^n c_{j\ell} (X_{j\ell}-1)$ which always holds since $W_k\ge 0$ and $X_{j\ell}\le 1$ for all $\ell=1,\ldots,n$.
Second, if $P_{jk}=1$ then the allocation cost of demand point $j$ goes in sorted position $k$ ($P_{jk}$ are the ordering variables) and then $W_k$ must be the allocation cost of the actual allocation of demand point $j$ to the open facility $\ell$, namely $W_k\ge \sum_{\ell=1}^k c_{j\ell} X_{j\ell}$.

Augmenting these constraints to \ref{pro:comb-ofp1}, it results exactly in \ref{pro:comb-ofp2} and the result follows.

 \hfill $\Box$
}

\end{proof}

Additionally and for the sake of readability we will introduce, in the sequel, the following notation.
We set the entire family of variables of the problem \ref{pro:comb-ofp2} as the column vector $\phi$ given by
\begin{equation} \label{def:phi}
 \phi^T=(rvec(P)^T,rvec(X)^T,\y^T,W^T,\xi^T,rvec(\eta)^T,\zeta^T).
\end{equation}
Recall that we assume that all single-index variables, namely $\y,W,\xi,\zeta$, are column vectors.
Furthermore, to alleviate the notation, we will denote the  vector of coefficients of a constraint numbered as $(\#)$, in the formulation \ref{pro:comb-ofp2}, by $[a_{(\#)}]$. If the reference $(\#)$ defines, for instance, a set of constraints for all $\ell=1,\ldots,n$ then we refer to the $\ell-th$ constraint in that set by $[a_{(\#)}]_\ell$. For instance, $[a_{(\ref{pro:comb1p})}]_1$ refers to the first equation of constraints (\ref{pro:comb1p}) which is $\sum_{j=1}^n P_{j1}=1$.

Next, we consider the matrix form $\Phi=\phi\phi^T$ that is symmetric and its upper triangular part is given by:
\begin{equation} \label{eq:Phi}
\begin{tabular}{rl}
& \begin{tabular}{p{8mm}p{8mm}p{8mm}p{8mm}p{6mm}p{8mm}p{6mm}}
\hspace*{4mm} $n^2$ & \hspace*{1mm} $n^2$ & \hspace*{1mm}  $n$ & \hspace*{-2mm} $n$ & \hspace*{-4mm} $n$ & \hspace*{-4mm} $n^2$ & \hspace*{-5mm} $n$
\end{tabular} \\
$\Phi=$ & $\left[\begin{array}{ccccccc}
{\overbrace{Q}} & \overbrace{V} & \overbrace{P\y} & \overbrace{PW} & \overbrace{P\xi} & \overbrace{P\eta} & \overbrace{P\zeta} \\
 & U &X\y & XW & X\xi & X\eta & X \zeta \\
 & & \Sigma & \y W & \y \xi & \y \eta & \y \zeta \\
 & & &\Omega & W \xi & W \eta & W \zeta \\
 & & & & \Psi & \xi \eta & \xi \zeta \\
 & & & & & \Pi & \eta \zeta  \\
 & & & & & & \mathcal{Z}
\end{array}\right] \in \mathbb{R}^{(3n^2+4n)\times (3n^2+4n)}$
\end{tabular}
\end{equation}

Each entry in $\Phi$ is a block of variables.
The different  submatrices in $\Phi$ are described below:

\rev{$$Q= rvec(P) rvec(P)^T=\left[ \begin{array}{c}
rvect(P) P_{11}\\
\vdots \\
rvect(P) P_{nn}
\end{array} \right]=(q_{ikj\ell})\in \mathbb{R}^{n^2\times n^2}.$$
There exist structural conditions on the product of variables $P_{ik}P_{j\ell}$ since $diag(Q)=rvect(P)$.}

The matrix $U$ is analogous replacing $P$ by $X$. Thus,

\rev{$$U= rvec(X) rvec(X)^T=\left[ \begin{array}{c}
rvect(X) X_{11}\\
\vdots \\
rvect(X) X_{nn}
\end{array} \right]=(u_{ikj\ell})\in \mathbb{R}^{n^2\times n^2}.$$
There also exist structural conditions on the product of variables $X_{ik}X_{j\ell}$ since $diag(U)=rvect(X)$.}

Next, the matrix $V$ is

\rev{$$V= rvec(P) rvec(X)^T=\left[ \begin{array}{c}
rvect(X) P_{11}\\
\vdots \\
rvect(X) P_{nn}
\end{array} \right]=(v_{ikj\ell})\in \mathbb{R}^{n^2\times n^2}.$$}
The remaining submatrices that are required because their terms appear in some constraints are defined accordingly:
$$\begin{array}{ll}
 \Sigma= \y\y^T=(\sigma_{ij})\in \mathbb{R}^{n\times n}, & \\ 
 \Omega=WW^T=(\omega_{ij})\in \mathbb{R}^{n\times n}, & \Psi=\xi\xi^T=(\psi_{ij})\in \mathbb{R}^{n\times n},\\
  \Pi=rvec(\eta)rvec(\eta)^T=(\pi_{ikj\ell})\in \mathbb{R}^{n^2\times n^2}, & \mathcal{Z}=\zeta\zeta^T=(z_{rs})\in \mathbb{R}^{n\times n}, \\
 (P\y)=rvec(P)\y^T=((P\y)_{ij\ell})\in \mathbb{R}^{n^2\times n}, & (PW)=rvec(P)W^T=(\rho_{ikj})\in \mathbb{R}^{n^2\times n}, \\
 (P\xi)=rvec(P)\xi^T=((P\xi)_{ij\ell})\in \mathbb{R}^{n^2\times n}, & (P\eta)= rvec(P)rvec(\eta)^T=(\nu_{ijrs})\in \mathbb{R}^{n^2\times n^2},\\
 (P\zeta)=rvec(P)\zeta^T=((P\zeta)_{ij\ell})\in \mathbb{R}^{n^2\times n}, & \\
 (X\y)=rvec(X)\y^T=(\chi_{ikj})\in \mathbb{R}^{n^2\times n},& (XW)=rvec(X)W^T=(\gamma_{ik\ell})\in \mathbb{R}^{n^2\times n},\\
 (X\eta) =rvec(X)rvec(\eta)^T=(\kappa_{ijrs})\in \mathbb{R}^{n^2\times n^2}, &  (X\zeta)=rvec(X)\zeta^T=(\tau_{ijr})\in \mathbb{R}^{n^2\times n},\\
 (\y\zeta)= \y\zeta^T=(\beta_{ij})\in \mathbb{R}^{n\times n},& \\
 (W\xi)=W\xi^T=(\delta_{ij})\in \mathbb{R}^{n\times n}, & (W\eta)= Wrvec(\eta)^T=(\epsilon_{irs})\in \mathbb{R}^{n\times n^2}.
\end{array}$$
\medskip

Finally, consider the coefficient matrix $G$ and the vector $g$ indexed in the same basis as $\phi$.
\medskip
$$ \begin{array}{ccc}
G=\left[ \begin{array}{ccccccc}
D & F & \Theta & \Theta & \Theta & \Theta & \Theta \\
F^T & H & \Theta & \Theta & \Theta & \Theta & \Theta \\
\Theta & \Theta & \Theta & \Theta & \Theta & \Theta & \Theta \\
\Theta & \Theta & \Theta & \Theta & \Theta & \Theta & \Theta \\
\Theta & \Theta & \Theta & \Theta & \Theta & \Theta & \Theta \\
\Theta & \Theta & \Theta & \Theta & \Theta & \Theta & \Theta \\
\Theta & \Theta & \Theta & \Theta & \Theta & \Theta & \Theta \end{array} \right] & \mbox{ where } F= &
\begin{array}{c}
 \left[\begin{array}{cccc}
\stackrel{X_{1\cdot}}{\overbrace{\lambda_1\otimes C_{1\cdot}}}& \stackrel{X_{2\cdot}}{\overbrace{\Theta}} & \ldots & \stackrel{X_{n\cdot}}{\overbrace{\Theta}}\\
\vdots & \vdots & \ddots &\vdots \\
\lambda_n\otimes C_{1\cdot} & \Theta & \ldots & \Theta \\
\Theta &\lambda_1\otimes C_{2\cdot} & \ldots & \Theta \\
\vdots &  \vdots & \ddots & \vdots \\
\Theta & \lambda_n\otimes C_{2\cdot} & \ldots & \Theta \\
\Theta & \Theta & \ldots & \lambda_1\otimes C_{n\cdot} \\
\vdots &  \vdots & \ddots & \vdots \\
\Theta & \Theta & \ldots & \lambda_n\otimes C_{n\cdot}
\end{array} \right]
\end{array}
\end{array}
$$
\medskip
and
$$ g^T=[\; \stackrel{P}{\overbrace{\Theta}},\stackrel{X}{\overbrace{\Theta}},\stackrel{\y}{\overbrace{\Theta}},
\stackrel{W}{\overbrace{\lambda^T}},\stackrel{\xi}{\overbrace{\Theta}},\stackrel{\eta}{\overbrace{\Theta}},\stackrel{\zeta}{\overbrace{\Theta}}\; ].$$
\rev{Recall that, as already introduced in Section \ref{s:codomp}, $\Theta$ stands for the matrix of the adequate size with all its entries equal to zero.}
\medskip

Clearly,
\begin{align}
1/2 \langle G, \Phi \rangle=1/2trace(G\Phi)& =\displaystyle \sum_{k=1}^n \lambda_k \sum_{j=1}^n P_{jk}\sum_{\ell =1} ^n c_{j \ell} X_{j\ell} +1/2\sum_{j=1}^n\sum_{k=1}^n\sum_{j'=1}^n\sum_{k'=1}^n  d_{jkj'k'}P_{jk}P_{j'k'} \nonumber \\
& +1/2\sum_{j=1}^n\sum_{\ell=1}^n\sum_{p=1}^n\sum_{q=1}^n  h_{j\ell pq}X_{j\ell}X_{pq}, \label{eq:ofp2}
\end{align}
is the objective function of Problem \ref{pro:comb-ofp2}. Analogously,
\begin{equation} \label{eq:of-con-g}
 \langle F, V\rangle=\rev{trace(FV)=} \sum_{k=1}^n \lambda_k \sum_{j=1}^n P_{jk}\sum_{\ell =1} ^n c_{j \ell} X_{j\ell}=\; g^T W,
\end{equation}
that results in the first term in the objective function of the same problem, namely, \ref{pro:comb-ofp2}.

Let $\mathcal{L}$ be \rev{the linear constraints} of the feasible region of \ref{pro:comb-ofp2}, namely
\begin{equation} \label{eq:L}
\rev{\mathcal{L}=\{\phi\ge 0: \phi \mbox{ satisfies } \eqref{pro:comb1p},\eqref{pro:comb2p}, \eqref{pro:comb4p}, \eqref{pro:comb5p}, \eqref{pro:comb6pmodeq}, 
\eqref{pro:comb1p8}, \eqref{pro:comb1p3} \mbox{ and }  \eqref{pro:comb1p7}\},}
\end{equation}
and let $A\phi=b$ be system of equations that describes the set $\mathcal{L}$.
Now, we consider a new transformation of \ref{pro:comb-ofp2} using matrix variables $\Phi$. \rev{This transformation requires, in addition, to include:} 1)  the linear constraints, $\mathcal{L}$, that come from \ref{pro:comb-ofp2}, namely $A\phi=b$, 2) the squares of those constraints in the matrix variable $\Phi$, namely $diag(A\Phi A^T)=b\circ b$, where $\circ$ is the Hadamard product of vectors, 3) the quadratic constraints of \ref{pro:comb-ofp2} written in the matrix variables $\Phi$:
\begin{align}
W_k-\sum_{\ell=1}^n\sum_{j=1}^n c_{j\ell} v_{jkj\ell}&=0,\; \forall k=1,\ldots,n, \label{quadratic1}\\
\sum_{j,k=1}^n (P_{jk}-q_{jkjk})&=0, \label{quadratic2}\\
\rev{\sum_{j,\ell=1}^n (X_{j\ell}-u_{j\ell j\ell})}&=0 \label{quadratic3}
\end{align}
and 4) \rev{the matrix
 $$\overline \Phi=\left( \begin{array}{cc} 1 & \phi^T \\ \phi & \Phi \end{array}\right) \in \mathcal{C}^*,$$ the cone of completely positive matrices of the appropriate dimension (see \citep[Theorem 3.1]{Burer09}).} Schematically, we can write that formulation as follows:

\begin{align}
 \min  &  \; \displaystyle \sum_{j=1}^n\sum_{k=1}^n\sum_{j'=1}^n\sum_{k'=1}^n  d_{jkj'k'}q_{jkj'k'} +1/2\sum_{j=1}^n\sum_{\ell=1}^n\sum_{p=1}^n\sum_{q=1}^n  h_{j\ell pq}u_{j\ell pq}\nonumber \\
 & +1/2\sum_{k=1}^n \lambda_k \sum_{j=1}^n\sum_{\ell=1}^n c_{j\ell}v_{jkj\ell} \label{DOMP-QUAD} \tag{CP-DOMP-0} \\
\mbox{s.t. } & \phi \in \mathcal{L}  \nonumber\\
& diag(A^T\Phi A)= b\circ b,  \mbox{ where $A$ is the matrix of } \mathcal{L} \nonumber\\
& W_k-\sum_{\ell=1}^n\sum_{j=1}^n c_{j\ell} v_{jkj\ell}=0,\; \forall k=1,\ldots,n,\nonumber\\
& \sum_{j,k=1}^n (P_{jk}-q_{jkjk})=0, \nonumber \\
& \rev{\sum_{j,\ell=1}^n (X_{j\ell}-u_{j\ell j\ell})=0, \nonumber} \\
&\overline \Phi = \left(\begin{array}{cc} 1 & \phi^T \\ \phi & \Phi \end{array}\right)\in \mathcal{C}^* \nonumber
\end{align}

It is well-known that \ref{DOMP-QUAD} is always a relaxation of \ref{pro:comb-ofp2} and therefore it is also a relaxation for \DOMP. Our next result proves that actually it is  not a relaxation of \DOMP but a reformulation.
\begin{thm} \label{th:3.1}
\DOMP belongs to the class of continuous, convex, conic optimization problems. Problem \ref{DOMP-QUAD} is equivalent to \ref{pro:comb-ofp2}, i.e.: (i) \rev{they have equal optimal} objective value, (ii) if $(\phi^*,\Phi^*)$ is an optimal solution for Problem \ref{DOMP-QUAD} then $\phi^* $ is in the convex hull of optimal solutions of Problem \ref{pro:comb-ofp2}.
\end{thm}
\begin{proof} \rm

Clearly, the objective function of the problem \ref{DOMP-QUAD} can be written as a linear form of $\Phi$ (see (\ref{eq:ofp2})):
$$ 1/2\langle D,Q \rangle+1/2\langle H,U \rangle+ \langle F, V \rangle= 1/2 \langle G, \Phi \rangle.$$

Let $\mathcal{L}$ be the linear constraints of the feasible region of \ref{pro:comb-ofp2}, as defined in (\ref{eq:L}),
and let $\mathcal{L}_{\infty}$ be its recession cone.

First of all, one realizes that $P$, $\y$ and $X$ are bounded \rev{above since they satisfy \eqref{pro:comb1p}, \eqref{pro:comb4p}, \eqref{pro:comb5p}, respectively. Besides $W_k$, for all $k=1,\ldots,n$ are also bounded above since if we apply constraint (\ref{pro:comb1p7}), we get $\sum_{k=1}^n W_k=\sum_{j=1}^n\sum_{\ell=1}^n c_{j\ell} X_{j\ell}$. Hence, for each $k=1,\ldots,n$, $W_k\le n\max_{j,\ell} c_{j\ell}$. This proves that all variables $P,X,\y,W$ are nonnegative and bounded above. To prove that also the slack variables $\eta,\psi,\zeta$ are bounded we proceed as follows. We observe from \eqref{pro:comb1p8} that $\eta_{jk}\le W_k+2\sum_{j,\ell}c_{j\ell}$, for all $j,k=1,\ldots,n$ and thus variables $\eta$ are bounded. Analogously, from \eqref{pro:comb1p3}, we get that $\psi_k\le W_k+W_{k+1}$, for all $k=1,\ldots,n-1$, hence $\psi$ variables are bounded as well. Finally, using \eqref{pro:comb6pmodeq}, it follows that $\zeta_\ell \le (n-p+1)\y_\ell +\sum_{j=1}^n\sum_{\ell=1}^n X_{j\ell}$, for all $\ell=1,\ldots,n$, and thus using the boundedness of $\y$ and $X$ the slack variables $\zeta$ are also bounded.}

Therefore, it follows that the recession cone of the linear part of the feasible region of \ref{pro:comb-ofp2} is the zero vector.

Next, we consider the sets:
\begin{align*}
\mathcal{L}'=& \mathcal{L}\cap \{\phi: W_k-\sum_{\ell=1}^n\sum_{j=1}^n c_{j\ell} P_{jk}X_{j\ell}=0,\; k=1,\ldots,n, \, \sum_{j,k=1}^n (P_{jk}-P_{jk}^2)=0, \, \rev{\sum_{j,\ell=1}^n (X_{j\ell}-X_{j\ell}^2)=0} \}\\
(\mathcal{L}')^1=& \left\{ {1\choose \phi} {1\choose \phi}^T: \phi\in \mathcal{L}'\right\},\\
\mathcal{R}=& \left\{ \overline{\Phi}=\left(\begin{array}{cc} 1 & \phi^T \\ \phi & \Phi \end{array}\right)\in \mathcal{C}^*: A\phi=b,\;\;  diag(A\Phi A^T)=b\circ b \right\},\\
\mathcal{R}'=& \mathcal{R}\cap \left\{ \overline{\Phi}=\left(\begin{array}{cc} 1 & \phi^T \\ \phi & \Phi \end{array}\right):W_k-\sum_{\ell=1}^n\sum_{j=1}^n v_{jk\ell j}=0,\forall k=1,\ldots,n,\, \sum_{j,k=1}^n (P_{jk}-q_{jkjk})=0, \right. \\
           &  \rev{\sum_{j,\ell=1}^n (X_{j\ell}-u_{j\ell j\ell})=0 \Big\}.}
\end{align*}
Recall the $A\phi=b$ is the set of linear constraints that describe the set $\mathcal{L}$, $\mathcal{L}'$ is the  continuous relaxation of \ref{pro:comb-ofp2} and $\mathcal{R}'$ is the feasible region of \ref{DOMP-QUAD}.

We prove, that
$$  conv((\mathcal{L}')^1) = \mathcal{R}'.$$

The inclusion  $ conv((\mathcal{L}')^1)\subseteq \mathcal{R}'$ is clear. For the reverse inclusion, since \rev{$\mathcal{R}'\subset \mathcal{R}\subset \mathcal{C}^*$  then for any matrix $\overline \Phi\in \mathcal{R}'$,  it is known (see e.g. \cite{Burer09}) that there exists a representation as:
$$\overline{\Phi}=\left(\begin{array}{cc} 1 & \phi^T \\ \phi & \Phi \end{array}\right)=\sum_{r\in I} \mu^r {1\choose \phi^r} {1\choose \phi^r}^T + \sum_{r \in J} {0\choose \gamma^r} {0\choose \gamma^r}^T, 
$$
where  $I$ and $J$ are  finite sets of indices, $\mu^r\ge 0$ for all $r\in I$, $\sum_{r\in I} \mu^r=1$, $\phi^r\in \mathcal{L}$ for all $r\in I$ and $\gamma^r\in \mathcal{L}_\infty$ for all $r\in J$. Therefore, since the recession cone $\mathcal{L}_\infty$ is the zero vector, the representation above reduces to:
\begin{equation} \label{eq:represent}
\overline{\Phi}=\left(\begin{array}{cc} 1 & \phi^T \\ \phi & \Phi \end{array}\right)=\sum_{r\in I} \mu^r {1\choose \phi^r} {1\choose \phi^r}^T, 
\end{equation}
with $\mu^r\ge 0$ for all $r\in I$, $\sum_{r\in I} \mu^r=1$, and $\phi^r\in \mathcal{L}$ for all $r\in I$.}
\medskip

We claim that $\sum_{j,k=1}^n (P_{jk}-q_{jkjk})=0$, \rev{$\sum_{j,\ell=1}^n (X_{j\ell}-u_{j\ell j\ell})=0$} and $W_k-\sum_{\ell=1}^n\sum_{j=1}^n c_{j\ell} v_{jk\ell j}=0,\forall k=1,\ldots,n$ implies that
each $\phi^r$ with $r\in I$ satisfies $\sum_{j,k=1}^n (P_{jk}^r-(P_{jk}^r)^2)=0$, \rev{$\sum_{j,\ell=1}^n (X_{j\ell}^r-(X_{j\ell}^r)^2)=0$} and $W_k^r -\sum_{\ell=1}^n\sum_{j=1}^n c_{j\ell} P_{jk}^rX_{j\ell}^r=0,\; \forall \; k=1,\ldots,n$. This will complete the proof.

First of all, we observe that, by the representation (\ref{eq:represent}), it holds that: \rev{$P_{jk}=\sum_{r\in I} \mu^r P_{jk}^r,\; X_{j\ell}=\sum_{r\in I} \mu^r X_{j\ell}^r$,} $W_k=\sum_{r\in I} \mu^r W_k^r$,\;  $q_{jkjk}=\sum_{r\in I} \mu^r (P_{jk}^r)^2$ and $v_{jk\ell j}=\sum_{r\in I} \mu^r P_{jk}^r X_{\ell j}^r$; for all $j,k=1,\ldots,n$.

We begin with $\sum_{j,k=1}^n (P_{jk}-q_{jkjk})=0$ and replace the variables $P_{jk}$ and $q_{jkjk}$ by their representation in terms of (\ref{eq:represent}) which results in:
\begin{equation} \label{eq:pjkr} 
 0=\sum_{j=1}^n\sum_{k=1}^n \left(\sum_{r\in I} \mu^r P_{jk}^r-\sum_{r\in I} \mu^r (P_{jk}^r)^2\right)=\sum_{j=1}^n\sum_{k=1}^n \left(\sum_{r\in I} \mu^r (P_{jk}^r-(P_{jk}^r)^2)\right).
\end{equation}
\rev{It is clear that $P_{jk}^r\le 1$ and $X_{j\ell}^r\le 1$ for all $j,k,\ell=1,\ldots,n$. Indeed,  $P_{jk}=\sum_{r\in I} \mu^r P_{jk}^r$ with $P_{jk}^r\ge 0$ for all $r\in I$ and $P_{jk}\le 1$ since it satisfies \eqref{pro:comb1p} and \eqref{pro:comb2p}. Therefore $P_{jk}^r\le 1$, for all $j,k=1,\ldots,n$, then all the addends in (\ref{eq:pjkr}) are nonnegative and this implies that
\begin{equation}\label{eq:cond1}
P_{jk}^r-(P_{jk}^r)^2=0, \qquad \forall \; j,k=1,\ldots,n, \; r\in I.
\end{equation}
Next,  the proof for $\sum_{j,\ell=1}^n (X_{j\ell}^r-(X_{j\ell}^r)^2)=0$ is similar, but to prove that $X_{j\ell}^r\le 1$ we use the inequality \eqref{pro:comb5p} valid for $X_{j\ell}$ instead of \eqref{pro:comb1p} and \eqref{pro:comb2p}. This proves the first part of the claim.}

Next, we consider $W_k-\sum_{\ell=1}^n\sum_{j=1}^n c_{j\ell} v_{jk\ell j}=0$. Then, we replace $W_k$ and $v_{jk\ell j}$ by their representation in terms of (\ref{eq:represent}) which results in
\begin{equation}\label{eq:segcon}
0=\sum_{r\in I} \mu^r W_k^r-\sum_{\ell=1}^n\sum_{j=1}^n \sum_{r\in I} \mu^r c_{j\ell} P_{jk}^r X_{jk}^r=\sum_{r\in I}\mu^r \left( W_k^r- \sum_{\ell =1}^n \sum_{j=1}^n c_{j\ell} P_{jk}^rX_{\ell j}^r\right).
\end{equation}

Now, since $W_k^r$, $X_{\ell j}^r$ and $P_{jk}^r$ satisfy \eqref{pro:comb1p8} for all $r\in I$ and  $P_{jk}^r\ge 0$ for all $j,k$ and $r$, we multiply both sides of inequality \eqref{pro:comb1p8} \rev{by} $P_{jk}^r$. This results in:
$$ W_k^r P_{jk}^r\ge \sum_{\ell =1}^n c_{j\ell} X_{j\ell}^r P_{jk}^r-\sum_{\ell=1}^n c_{j\ell}P_{jk}^r (1-P_{jk}^r), \quad \forall j,k=1,\ldots,n,\; r\in I.$$
We sum the above inequalities for all $j=1,\ldots,n$ to obtain:
$$ W_k^r \displaystyle \stackrel{= 1\; (\ref{pro:comb1p}),(\ref{eq:represent})}{\overbrace{\displaystyle \sum_{j=1}^n P_{jk}^r}}\ge \sum_{j=1}^n \sum_{\ell =1}^n c_{j\ell} X_{j\ell}^r P_{jk}^r- \sum_{j=1}^n \sum_{\ell=1}^n c_{j\ell}\stackrel{=0 \, (\ref{eq:cond1})}{\overbrace{(P_{jk}^r-(P_{jk}^r)^2)}}, \quad \forall k=1,\ldots,n,\; r\in I.$$

Furthermore, in the above inequality the factor multiplying $W_k^r$ in the left-hand-side is  equal to one by (\ref{pro:comb1p}) and (\ref{eq:represent}), and the second term is the right-hand-side in zero by (\ref{eq:cond1}). Hence, we obtain:
\begin{equation}\label{eq:cond2}
W_k^r \ge \sum_{j=1}^n \sum_{\ell =1}^n c_{j\ell} X_{j\ell}^r P_{jk}^r, \quad \forall k=1,\ldots,n,\; r\in I.
\end{equation}
This condition combined with (\ref{eq:segcon}) proves the second part of the claim. Therefore, $ \mathcal{R}' \subseteq conv((\mathcal{L}')^1)$ also holds, and thus, $ \mathcal{R}'= conv((\mathcal{L}')^1)$.

The above property allows us to rewrite Problem \ref{pro:comb-ofp2} as an equivalent continuous, convex conic linear \rev{optimization} problem in the matrix variables ${1\quad \phi^T \choose \phi \quad \Phi}$ as described in \ref{DOMP-QUAD}.  \hfill $\Box$
\end{proof}

The next result states that in \ref{DOMP-QUAD} we can remove the explicit dependency on the original variables $\phi$ obtaining a new formulation that only depends on the essential matrix variables $\Phi$.

\rev{
\begin{thm} \label{th:3.2}
Problem \ref{DOMP-QUAD} can be reformulated removing the explicit dependency on the $\phi$ variables as:
\begin{align}
 \min  &  \; \displaystyle \sum_{j=1}^n\sum_{k=1}^n\sum_{j'=1}^n\sum_{k'=1}^n  d_{jkj'k'}q_{jkj'k'} +1/2\sum_{j=1}^n\sum_{\ell=1}^n\sum_{p=1}^n\sum_{q=1}^n  h_{j\ell pq}u_{j\ell pq}\nonumber \\
 & +1/2\sum_{k=1}^n \lambda_k \sum_{j=1}^n\sum_{\ell=1}^n c_{j\ell}v_{jkj\ell} \label{pro:comb-of} \tag{CP-DOMP} \\
\mbox{s.t. } & \Phi[a_{(4)}]_1 \in \mathcal{L}  \nonumber\\
& diag(A^T\Phi A)= b\circ b,  \mbox{ where $A$ is the matrix of } \mathcal{L} \nonumber\\
& \sum_{\ell = 1}^n\rho_{1\ell k} - \sum_{\ell = 1}^n\sum_{j=1}^n c_{j\ell} v_{jkj\ell}=0,\; \forall k=1,\ldots,n,\nonumber\\
& \sum_{j,k=1}^n (\sum_{\ell = 1}^n q_{jk1\ell} -q_{jkjk})=0, \nonumber \\
& \sum_{j,\ell=1}^n (\sum_{r=1}^n v_{j\ell 1r} - u_{j\ell j\ell})=0, \nonumber \\
& \Phi \in \mathcal{C}^* \nonumber
\end{align}

Moreover, if $\hat \Phi$  is an optimal solution for Problem \ref{pro:comb-of} then $\hat \Phi [a_{(\ref{pro:comb1p})}]_1$ is in the convex hull of optimal solutions of Problem \ref{pro:comb-ofp2}.
\end{thm}
}
\begin{proof}
{\rm
\rev{The proof is based on finding an appropriate linear combination of the rows of the linear constraints in $\mathcal{L}$, \rev{namely $A\phi=b$,} that allows one a writing of $\phi$ as a linear combination of $\Phi$, following an argument similar to the one in \cite{Burer09}. We reproduce it for the sake of completeness.

Take  the first equation in (\ref{pro:comb1p}) i.e., $\sum_{j=1}^n P_{j1}=1$. Recall that we rewrite this constraints as $[a_{(\ref{pro:comb1p})}]^T_1 \phi=1$.  Next, define the vector $\beta=(\beta_i)_{i\in I}$, as $\beta_1=1$, $\beta_i=0$, for all $i>1$.

It is clear that $\sum_{i\in I} \beta_i [a_{(i)}]=[a_{(\ref{pro:comb1p})}]_1\ge 0$ and that $\sum_{i\in I} \beta_i b_i=1$. Let us denote $\alpha=\sum_{i\in I} \beta_i [a_{(i)}]$. Then,
$$ \alpha^T \phi =\sum_{i\in I} \beta_i [a_{(i)}]^T \phi = \sum_{i\in I} \beta_i b_i=1.$$
Thus, since $\phi \phi^T=\Phi$, from the above we obtain 
\begin{equation}\label{filineal}
 \phi \phi^T \alpha= \Phi \alpha = \phi.
\end{equation}

\noindent On the other hand, $1=\alpha^T \phi \phi^T \alpha= \alpha^T \Phi  \alpha$, and hence
$$ \overline{\Phi}:= \left[\begin{array}{cc} 1 & \phi^T \\ \phi & \Phi \end{array} \right]= \left[\begin{array}{cc} 1 & \alpha^T \Phi \\ \Phi \alpha & \Phi \end{array} \right] = (\alpha \; I)^T \Phi (\alpha I),$$
and since $\alpha\ge 0$, if $\Phi \in \mathcal{C}^*$ then $\overline{\Phi} \in \mathcal{C}^*$.
To prove the converse, it suffices to recall that the principal submatrices of completely positive matrices are completely positive. Therefore, if $\overline{\Phi} \in \mathcal{C}^*$ then $\Phi \in \mathcal{C}^*$. 

To obtain the formulation (\ref{pro:comb-of}) we only need to replace any occurrence of $\phi$ in \ref{DOMP-QUAD} by $\Phi [a_{(\ref{pro:comb1p})}]_1$ and observe that, in terms of $\Phi [a_{(\ref{pro:comb1p})}]_1$, $W_k=\sum_{\ell =1}^n \rho_{1\ell k}$, for all $k=1,\ldots,n$, $P_{jk}=\sum_{\ell = 1}^n q_{jk1\ell}$, for all $j,k=1,\ldots,n$ and $X_{j\ell}=\sum_{r=1}^n v_{j\ell 1r}$, for all $j,\ell =1,\ldots,n$.

To complete the proof we apply the assertion (ii) of Theorem \ref{th:3.1}.

Summarizing, we have obtained that formulation \ref{pro:comb-of} is exact for \DOMP and it is linear in $\Phi \in \mathcal{C}^*_{(3n^2+4n)\times (3n^2+4n)}$. }
 \hfill $\Box$
}
\end{proof}


An explicit reformulation of \ref{pro:comb-of}, namely (\ref{pro:explicit}),  can be found in the Appendix.

\begin{propo} \label{cor:conv-of}
The objective function of \ref{pro:comb-of} can be equivalently written for any $\mu \in [0,1]$ as:
\begin{equation} \label{eq:cor31}
\mu \langle F,V \rangle+(1-\mu) \sum_{k=1}^n \lambda_k \sum_{\ell=1}^n \rho_{1\ell k}+1/2\langle D,Q \rangle+1/2\langle H,U \rangle.
\end{equation}
\end{propo}
\begin{proof} \rm
Recall that from (\ref{eq:ofp2}) the objective function of \ref{pro:comb-ofp2}satisfies:
\begin{align*}
\displaystyle \sum_{k=1}^n \lambda_k \sum_{j=1}^n P_{jk}\sum_{\ell =1} ^n c_{j \ell} X_{j\ell} +1/2\sum_{j=1}^n\sum_{k=1}^n\sum_{j'=1}^n\sum_{k'=1}^n  d_{jkj'k'}P_{jk}P_{j'k'} \\
+1/2\sum_{j=1}^n\sum_{\ell=1}^n\sum_{p=1}^n\sum_{q=1}^n  h_{j\ell pq}X_{j\ell}X_{pq} & =1/2 \langle G, \Phi \rangle\\
 &\hspace*{-2cm} =  \langle F,V\rangle+ 1/2 \langle D,Q\rangle +1/2\langle H,U\rangle.
\end{align*}
Next, using (\ref{eq:of-con-g}) we get
$$  \langle F,V \rangle= \sum_{k=1}^n \lambda_k\sum_{j=1}^n\sum_{\ell=1}^n c_{j\ell} P_{jk}X_{j\ell}=\sum_{k=1}^n \lambda_k W_k=\langle g, \phi \rangle.$$
\rev{Then, for any $\mu\in [0,1]$, combining both expressions and using that $\Phi\alpha=\phi$ and therefore $W_k=\sum_{\ell=1}^n\rho_{1\ell k}$, for all $k=1,\ldots, n$, we get:}
{\small
\begin{align}
\langle F, V \rangle+1/2\langle D,Q \rangle+1/2\langle H,U \rangle=& \mu \langle F, V \rangle +(1-\mu) \langle g, \phi \rangle+1/2\langle D,Q \rangle+1/2\langle H,U \rangle \nonumber \\
= & 1/2 \langle D,Q \rangle+1/2\langle H,U \rangle+\mu \langle F, V \rangle +(1-\mu) \langle \mu, W \rangle \nonumber \\
=& 1/2 \langle D,Q \rangle+1/2\langle H,U \rangle+\mu \langle F, V \rangle \nonumber \\
& +(1-\mu) \langle \mu, (PW)^T [a_{(\ref{pro:comb1p})}]_1 \rangle \nonumber   \\
=&  1/2 \langle D,Q \rangle+1/2\langle H,U \rangle+\mu/2 \langle F, V \rangle \nonumber \\
&+(1-\mu) \sum_{k=1}^n \lambda_k \sum_{\ell=1}^n \rho_{1\ell k}. \label{pro:miqlc-of-ref2}
\end{align}
}
This is exactly the expression that we had to obtain.
 \hfill $\Box$
\end{proof}

Finally, we would like to remark that 7 out of 21 blocks of the matrix variables $\Phi$, namely $\y W\in \mathbb{R}^{n \times n}$, $\y \xi\in \mathbb{R}^{n \times n}$, $\y \eta\in \mathbb{R}^{n \times n^2}$, $ W \zeta\in \mathbb{R}^{n \times n}$, $\xi \eta \in \mathbb{R}^{n \times n^2}$, $\xi \zeta\in \mathbb{R}^{n \times n}$ and $\eta \zeta\in \mathbb{R}^{n^2 \times n}$,  never appear explicitly in any constraint but in $\Phi \in \mathcal{C}^*_{(3n^2+4n)\times (3n^2+4n)}$.
In addition, we observe that  Corollary \ref{cor:conv-of} shows that, according to (\ref{eq:cor31}), the objective function of \ref{pro:comb-of} is also valid for $\mu=0$ and $\mu=1$ which makes it possible to avoid $V$ or $\rho$ variables in the objective function. It is not clear whether these two final remarks may help in solving the problem or not.

\section{Concluding remarks\label{s:remarks}}
The results in this paper state, for the first time, the equivalence of a difficult $NP$-hard discrete location problem, namely \DOMP, with a continuous, convex problem. This new approach can be used to start new avenues of research by applying tools from continuous optimization to approximate or numerically solve some other hard discrete location problems, as for instance single allocation hub location problems and ordered median hub location problems with and without capacities, see e.g.  \cite{FPRCh2013,Puerto2011,PUERTO2016142}, to mention a few. The aim of this paper is not computational but it seems natural to consider some relaxations of  formulation \ref{pro:comb-of} to analyze the accuracy of their related bounds. This is beyond the scope of this contribution but will be the subject of a follow up paper.

\section*{Acknowledgements}

This research has been partially supported by Spanish Ministry of Econom{\'\i}a and  Competitividad/FEDER grants number
MTM2016-74983-C02-01. The author would like to thank Prof. I. Bomze from the University of Vienna for fruitful discussions that led to conclude this project.

\bibliographystyle{plainnat}

\begin{thebibliography}{20}
\providecommand{\natexlab}[1]{#1}
\providecommand{\url}[1]{\texttt{#1}}
\expandafter\ifx\csname urlstyle\endcsname\relax
  \providecommand{\doi}[1]{doi: #1}\else
  \providecommand{\doi}{doi: \begingroup \urlstyle{rm}\Url}\fi

\bibitem[Bai et~al.(2016)Bai, Mitchell, and Pang]{Bai2016}
L.~Bai, J.~E. Mitchell, and J.-S. Pang.
\newblock On conic qpccs, conic qcqps and completely positive programs.
\newblock \emph{Mathematical Programming}, 159\penalty0 (1):\penalty0 109--136,
  Sep 2016.

\bibitem[Blanco et~al.(2014)Blanco, Ali, and Puerto]{BHP2014}
V.~Blanco, S.~El Haj~Ben Ali, and J.~Puerto.
\newblock Revisiting several problems and algorithms in continuous location wih
  $l_\tau-$norm.
\newblock \emph{Computational Optimization and Applications}, 58\penalty0
  (3):\penalty0 563--595, 2014.

\bibitem[Boland et~al.(2006)Boland, Dom{\'\i}nguez-Mar{\'\i}n, Nickel, and
  Puerto]{Boland2006}
N.~Boland, P.~Dom{\'\i}nguez-Mar{\'\i}n, S.~Nickel, and J.~Puerto.
\newblock Exact procedures for solving the discrete ordered median problem.
\newblock \emph{Computers \& Operations Research}, 33\penalty0 (11):\penalty0
  3270--3300, 2006.

\bibitem[Bomze(2012)]{Bomze12}
I.~M. Bomze.
\newblock Copositive optimization - recent developments and applications.
\newblock \emph{European Journal of Operational Research}, 216\penalty0
  (3):\penalty0 509--520, 2012.

\bibitem[Burer(2009)]{Burer09}
S.~Burer.
\newblock On the copositive representation of binary and continuous nonconvex
  quadratic programs.
\newblock \emph{Math. Program.}, 120\penalty0 (2):\penalty0 479--495, 2009.

\bibitem[Burer(2012)]{Burer2012}
S.~Burer.
\newblock \emph{Copositive Programming}, chapter~8, pages 201--218.
\newblock Springer, 2012.
\newblock ISBN 978-1-4614-0769-0.
\newblock Handbook on Semidefinite, Conic and Polynomial Optimization, M.F.
  Anjos and J.B. Lasserre(Eds.).

\bibitem[Burer(2015)]{Burer2015}
S.~Burer.
\newblock A gentle, geometric introduction to copositive optimization.
\newblock \emph{Math. Program.}, 151\penalty0 (1):\penalty0 89--116, 2015.

\bibitem[Burer and Dong(2012)]{BURER2012203}
S.~Burer and H.~Dong.
\newblock Representing quadratically constrained quadratic programs as
  generalized copositive programs.
\newblock \emph{Operations Research Letters}, 40\penalty0 (3):\penalty0 203 --
  206, 2012.


\bibitem[Burkard(1984)]{BURKARD1984283}
R.E. Burkard
\newblock Quadratic assignment problems.
\newblock \emph{European Journal of Operational Research}, 15\penalty0 (3):\penalty0 283--289, 1984.


\bibitem[D{\"u}r(2010)]{Dur10}
M.~D{\"u}r.
\newblock \emph{Copositive Programming -- a Survey}.
\newblock Springer, Berlin, Heidelberg, 2010.
\newblock ISBN 978-3-642-12597-3.
\newblock In: Diehl M., Glineur F., Jarlebring E., Michiels W. (eds) Recent
  Advances in Optimization and its Applications in Engineering.

\bibitem[Fern\'{a}ndez et~al.(2013)Fern\'{a}ndez, Puerto, and
  Rodr\'{\i}guez-Ch\'{\i}a]{FPRCh2013}
E.~Fern\'{a}ndez, J.~Puerto, and A.~M. Rodr\'{\i}guez-Ch\'{\i}a.
\newblock On discrete optimization with ordering.
\newblock \emph{Annals of Operations Research}, 207\penalty0 (1):\penalty0
  83--96, 2013.

\bibitem[Kalcsics et~al.(2003)Kalcsics, Nickel, and Puerto]{Kalcsics2003}
J.~Kalcsics, S.~Nickel, and J.~Puerto.
\newblock Multifacility ordered median problems on networks: A further
  analysis.
\newblock \emph{Networks}, 41\penalty0 (1):\penalty0 1--12, 2003.

\bibitem[Labb\'e et~al.(2017)Labb\'e, Ponce, and Puerto]{Labbe2016}
M.~Labb\'e, D.~Ponce, and J.~Puerto.
\newblock A comparative study of formulations and solution methods for the
  discrete ordered p-median problem.
\newblock \emph{Computers \& Operations Research}, 78:\penalty0 230 -- 242,
  2017.

\bibitem[Mar{\'{\i}}n et~al.(2009)Mar{\'{\i}}n, Nickel, Puerto, and
  Velten]{Marin2009}
A.~Mar{\'{\i}}n, S.~Nickel, J.~Puerto, and S.~Velten.
\newblock A flexible model and efficient solution strategies for discrete
  location problems.
\newblock \emph{Discrete Applied Mathematics}, 157\penalty0 (5):\penalty0
  1128--1145, 2009.

\bibitem[Nickel(2001)]{Nickel2001}
S.~Nickel.
\newblock Discrete ordered weber problems.
\newblock In \emph{Operations Research Proceedings 2000}, pages 71--76.
  Springer Verlag, 2001.

\bibitem[Nickel and Puerto(2005)]{Nickel2005}
S.~Nickel and J.~Puerto.
\newblock \emph{Location Theory: A Unified Approach.}
\newblock Springer Verlag, 2005.

\bibitem[Pe{\~{n}}a et~al.(2015)Pe{\~{n}}a, Vera, and Zuluaga]{Pena2015}
J.~Pe{\~{n}}a, J.~C. Vera, and L.~F. Zuluaga.
\newblock Completely positive reformulations for polynomial optimization.
\newblock \emph{Mathematical Programming}, 151\penalty0 (2):\penalty0 405--431,
  Jul 2015.

\bibitem[Puerto and Fern{\'a}ndez (1999)]{Puerto1999}
J.~Puerto and F.~R.~Fern{\'a}ndez. 
\newblock Multi-criteria minisum facility location problems.
\newblock \emph{Journal of Multi-Criteria Decision Analysis.} 8\penalty0 (5): \penalty0 268--280, 1999.

\bibitem[Puerto et~al.(2011)Puerto, Ramos, and
  Rodr\'{\i}guez-Ch\'{\i}a]{Puerto2011}
J.~Puerto, A.~B. Ramos, and A.~M. Rodr\'{\i}guez-Ch\'{\i}a.
\newblock Single-allocation ordered median hub location problems.
\newblock \emph{Computers \& Operations Research.}, 38\penalty0 (2):\penalty0
  559--570, feb 2011.

\bibitem[Puerto et~al.(2016)Puerto, Ramos, Rodr\'{\i}guez-Ch\'{\i}a, and
  S\'anchez-Gil]{PUERTO2016142}
J.~Puerto, A.~B. Ramos, A.~M. Rodr\'{\i}guez-Ch\'{\i}a, and M.~C. S\'anchez-Gil.
\newblock Ordered median hub location problems with capacity constraints.
\newblock \emph{Transportation Research Part C: Emerging Technologies},
  70:\penalty0 142 -- 156, 2016.

\bibitem[Stanimirovic et~al.(2007)Stanimirovic, Kratica, and
  Dugosija]{Stanimirovic2005}
Z.~Stanimirovic, J.~Kratica, and D.~Dugosija.
\newblock Genetic algorithms for solving the discrete ordered median problem.
\newblock \emph{European Journal of Operational Research}, 182:\penalty0
  983--1001, 2007.

\bibitem[{Xu} and {Burer}(2016)]{XuBurer2016}
G.~{Xu} and S.~{Burer}.
\newblock {A Copositive Approach for Two-Stage Adjustable Robust Optimization
  with Uncertain Right-Hand Sides}.
\newblock \emph{ArXiv e-prints}, September 2016.

\end{thebibliography}

\section*{Appendix: An explicit formulation of \ref{pro:comb-of}}
\rev{One can check that
\begin{align*}
\Phi [a_{(\ref{pro:comb1p})}]_1 = & \Big[\stackrel{(PP)}{\overbrace{\sum_{\ell=1}^n q_{111\ell},\ldots, \sum_{\ell=1}^n q_{nn1\ell}},} \;
\stackrel{(PX)}{\overbrace{\sum_{\ell=1}^n v_{111\ell},\ldots, \sum_{\ell=1}^n v_{nn1\ell}},} \;
\stackrel{(P\y)}{\overbrace{\sum_{\ell=1}^n (P\y)_{1\ell 1}, \ldots, \sum_{\ell=1}^n (P\y)_{1\ell n}},}  \\
& \stackrel{(PW)}{\overbrace{\sum_{\ell=1}^n (PW)_{1\ell 1}, \ldots, \sum_{\ell=1}^n (PW)_{1\ell n}},}\;
\stackrel{(P\xi)}{\overbrace{\sum_{\ell=1}^n (P\xi)_{1\ell 1}, \ldots, \sum_{\ell=1}^n(P\xi)_{1\ell n}},}\;
\stackrel{(P\eta)}{\overbrace{\sum_{\ell=1}^n (P\eta)_{1\ell 11}, \ldots, \sum_{\ell=1}^n (P\eta)_{1\ell nn}},}\; \\
& \stackrel{(P\zeta)}{\overbrace{\sum_{\ell=1}^n (P\zeta)_{1\ell 1}, \ldots, \sum_{\ell=1}^n (P\zeta)_{1\ell n}}}\Big]^T.
\end{align*}

Next, we obtain the explicit formulation of \ref{pro:comb-of}, which is obtained from the original formulation replacing the original variables $\phi$ by its linear expression in terms of $\Phi$, namely $\phi=\Phi [a_{(\ref{pro:comb1p})}]_1$ (see theorems \ref{th:3.1}  and \ref{th:3.2}).


Indeed, this reformulation requires to include: 1)  the linear constraints that come from \ref{pro:comb-ofp2} rewritten using that $\Phi [a_{(\ref{pro:comb1p})}]_1=\phi$, 2) the squares of those constraints in the matrix variable $\Phi$, 3) the quadratic constraints written in the matrix variables $\Phi$ and 4)  $\Phi\in \mathcal {C}^*$, the cone of completely positive matrices of the appropriate dimension.

In the following we check that the four conditions mentioned above give us the constraints that appear in the explicit representation of \ref{pro:comb-of} included below. Indeed,

\begin{enumerate}
\item Constraints (\ref{pro:comb-2}) are $[a_{(\ref{pro:comb1p})}]^T_i \Phi [a_{(\ref{pro:comb1p})}]_1=1$ for all $i=1,\ldots,n$. Analogously, constraints (\ref{pro:comb-3}) are $[a_{(\ref{pro:comb2p})}]^T_i \Phi [a_{(\ref{pro:comb1p})}]_1=1$ for all $i=1,\ldots,n$; constraint (\ref{pro:comb-5}) is $[a_{(\ref{pro:comb4p})}]^T\Phi [a_{(\ref{pro:comb1p})}]_1=p$; constraints (\ref{pro:comb-6}) are $[a_{(\ref{pro:comb5p})}]^T_i \Phi [a_{(\ref{pro:comb1p})}]_1=1$ for all $i=1,\ldots,n$; constraints (\ref{pro:comb-9}) are $[a_{(\ref{pro:comb1p3})}]^T_k \Phi [a_{(\ref{pro:comb1p})}]_1=1$ for all $k=1,\ldots,n-1$; constraints (\ref{pro:comb-7mod}) are $[a_{(\ref{pro:comb6pmodeq})}]^T_{\ell} \Phi [a_{(\ref{pro:comb1p})}]_1=1$ for all $\ell=1,\ldots,n$;
    and constraints (\ref{pro:comb-11}) are $[a_{(\ref{pro:comb1p8})}]^T_{jk} \Phi [a_{(\ref{pro:comb1p})}]_1=1$ for all $j,k=1,\ldots,n$. This proves that the block $A\phi=A\Phi [a_{(\ref{pro:comb1p})}]_1=b$ appears in \ref{pro:explicit}.
\item Constraints (\ref{pro:comb1xcomb1})-(\ref{pro:comb8xcomb8}) are obtained squaring (\ref{pro:comb1p}),(\ref{pro:comb2p}), (\ref{pro:comb4p}),(\ref{pro:comb5p}), (\ref{pro:comb1p3}), 
     (\ref{pro:comb6pmodeq})  and  (\ref{pro:comb1p8}), respectively. This proves that $diag(A^T\Phi A)=b\circ b$ also appears in \ref{pro:explicit}.
\item Constraints (\ref{pro:comb-QUAD1}), (\ref{pro:comb-QUAD2}) and (\ref{pro:comb-QUAD3}) are the quadratic constraints (\ref{pro:comb1p2}), (\ref{pro:comb2p2}) and (\ref{pro:combq2p2})  replacing quadratic terms by the corresponding matrix variables in $\Phi$ (recall that $\Phi$ was introduced in (\ref{eq:Phi})). 
    Hence, \ref{pro:explicit} includes the quadratic constraints in \ref{pro:comb-ofp2} written in the terms of the matrix variables $\Phi$.
\item $\Phi \in \mathcal{C}^*_{(3n^2+4n)\times (3n^2+4n)}$, the appropriate dimension of the space of variables.
\end{enumerate}

We note in passing that we do not need to add the constraint $[a_{(\ref{pro:comb1p})}]^T_1 \Phi [a_{(\ref{pro:comb1p})}]_1=1$  since it is already included; indeed it is the first constraint in the block (\ref{pro:comb-2}).}

An explicit reformulation of \ref{pro:comb-of} as a completely positive convex problem in the essential matrix variables $\Phi$ is the following
\begin{align}
\min \quad &   \langle F, V \rangle+ 1/2 \langle D,Q \rangle+1/2\langle H,U \rangle   \label{pro:explicit} \tag{CP-DOMP-Explicit}\\
\mbox{s.t. } & \sum_{j=1}^n\sum_{\ell=1}^n q_{ij1\ell}=1, \quad \forall i=1,\ldots,n \label{pro:comb-2}\\
& \sum_{k=1}^n\sum_{\ell=1}^nq_{ki1\ell}=1, \quad \forall i=1,\ldots,n \label{pro:comb-3}\\
& \sum_{k=1}^n\sum_{\ell=1}^n (P\y)_{1k\ell}=p , \quad \label{pro:comb-5}\\
& \sum_{j=1}^n\sum_{\ell= 1}^n v_{ij1\ell}=1, \quad \forall i=1,\ldots,n \label{pro:comb-6}\\
& \sum_{j=1}^n\sum_{k=1}^n v_{1kj\ell}-(n-p+1)\sum_{k=1}^n (P\y)_{1k\ell}+\sum_{k=1}^n (P\zeta)_{1k\ell}=0, \quad \forall \ell=1,\ldots,n \label{pro:comb-7mod} \\
& \sum_{\ell=1}^n \rho_{1\ell k}-\sum_{\ell=1}^n \rho_{1\ell k+1}+\sum_{\ell= 1}^n (P\xi)_{1\ell k}=0, \quad \forall k=1,\ldots, n-1 \label{pro:comb-9}\\
& \sum_{\ell=1}^n \rho_{1\ell k} -\sum_{\ell=1}^n(\sum_{r=1}^n v_{1rj\ell})c_{j\ell}+\sum_{\ell=1}^n c_{j\ell} (1-\sum_{r=1}^n q_{jk1r})  \nonumber \\
& +\sum_{r=1}^n (P\eta)_{1rjk}=0, \quad \forall j,k=1,\ldots,n \label{pro:comb-11}\\
& \sum_{i=1}^n q_{ikik}=1, \quad \forall k=1,\ldots,n \label{pro:comb1xcomb1}\\
&\sum_{i=k}^n q_{ikik}=1, \quad \forall i=1,\ldots,n \label{pro:comb2xcomb2}\\
&\displaystyle  \sum_{i=1}^n\sum_{j=1}^n\sigma_{ij}=p^2 \label{pro:comb4xcomb4}\\
& \displaystyle \sum_{k=1}^n\sum_{l=1}^n u_{jkj\ell}=1, \quad \forall j=1,\ldots,n \label{pro:comb5xcomb5}\\
& \omega_{kk}-2\omega_{k,k+1}+2\delta_{kk}+\omega_{k+1,k+1}-2\delta_{k,k+1}+\psi_{kk}=0, \quad \forall k=1,\dots,n-1 \label{pro:comb3xcomb3}\\
& \displaystyle \sum_{r=1}^n\sum_{s=1}^n u_{r\ell s\ell}-2(n-p+1)\sum_{r=1}^n \chi_{r\ell\ell}+2 \sum_{r=1}^n \tau_{r \ell \ell}-2(n-p+1)\beta_{\ell \ell} \nonumber \\
&\displaystyle  +(n-p+1)^2\sigma_{\ell\ell}+z_{\ell\ell}=0, \quad \forall \ell=1,\dots,n \label{pro:comb6xcomb6} \\
&  \displaystyle \sum_{r=1}^n\sum_{s=1}^n c_{rs} u_{rsrs}+\sum_{\scriptsize \begin{array}{c} i,j,r,s=1\\ (i,j)\neq (r,s)\end{array}}^n  c_{ij} c_{rs} u_{ijrs}-2\sum_{i,j=1} n c_{ij} (\sum_{\ell =1}^n \gamma_{ij\ell})+\sum_{r,s=1} ^n\omega_{rs}=0 \label{pro:comb7xcomb7}\\
&\displaystyle  (\sum_{\ell=1} ^n c_{j\ell})^2 q_{jkjk}+ \sum_{\ell=1} ^n c_{j\ell}^2 u_{j\ell j\ell}-2 (\sum_{\ell=1} ^n c_{j\ell})\sum_{\ell=1}^nc_{j\ell} v_{jkj\ell}\nonumber \\
&\displaystyle  +2 \sum_{\ell=1}^nc_{j\ell} \rho_{jkk}+2(\sum_{\ell=1}^nc_{j\ell})\nu_{jkjk}+2\sum_{r<s}^n c_{jr} c_{js} u_{jrjs}-2\sum_{\ell=1}^nc_{j\ell}\gamma_{j\ell k} \nonumber \\
&\displaystyle  -2\sum_{\ell=1}^n c_{j\ell}\kappa_{j\ell jk}+2\epsilon_{kjk}+\omega_{kk}+{\pi}_{jkjk}= (\sum_{\ell=1} ^n c_{j\ell})^2, \quad \forall j,k=1,\ldots,n \label{pro:comb8xcomb8}\\
& \displaystyle \sum_{\ell=1}^n \rho_{1\ell k} -\sum_{\ell=1}^n\sum_{j=1}^n c_{j\ell} v_{jkj\ell }=0, \quad \forall k=1,\ldots,n, \label{pro:comb-QUAD1}\\
& \sum_{j,k=1}^n (\sum_{\ell=1}^n q_{1\ell jk}-q_{jkjk})=0  \label{pro:comb-QUAD2} \\
& \sum_{j,\ell=1}^n (\sum_{r=1}^n v_{j\ell 1r} - u_{j\ell j\ell})=0 \label{pro:comb-QUAD3} \\
& \Phi \in \mathcal{C}^*_{(3n^2+4n)\times (3n^2+4n)} \label{pro:comb-cp}
\end{align}

\end{document}